\documentclass[letterpaper, 10 pt, conference]{ieeeconf}  

\IEEEoverridecommandlockouts                              

\usepackage[bottom]{footmisc}

\usepackage{amsthm}
\usepackage{amssymb}
\usepackage{mathrsfs}
\usepackage{amsmath}
\usepackage{mathtools}
\usepackage{enumerate}
\usepackage[dvipsnames]{xcolor}
\usepackage[noadjust]{cite}
\usepackage{indentfirst}
\usepackage{graphicx} 
\graphicspath{{./fig/}}
\usepackage[font=footnotesize]{caption}
\usepackage{subcaption}
\usepackage[breaklinks=true, colorlinks, bookmarks=true]{hyperref}
\usepackage{pstricks}
\usepackage{tikz}
\usepackage{pgfplots}
\pgfplotsset{compat=1.9}

\newtheorem{theorem}{Theorem}[section]
\newtheorem{lemma}[theorem]{Lemma}
\newtheorem{definition}[theorem]{Definition}
\newtheorem{corollary}[theorem]{Corollary}

\newtheorem{problem}{Problem}
\newtheorem{remark}[theorem]{Remark}

\newtheorem{example}[theorem]{Example}

\newcommand{\nom}{\operatorname{nom}}

\newcommand{\oprocendsymbol}{\hbox{$\bullet$}}
\newcommand{\oprocend}{\relax\ifmmode\else\unskip\hfill\fi\oprocendsymbol}

\newcommand{\stkout}[1]{\ifmmode\text{\sout{\ensuremath{#1}}}\else\sout{#1}\fi}

\definecolor{paleGreen}{rgb}{.3, .7, .3}

\newcommand{\until}[1]{\{1,\dots,#1\}}

\DeclareMathOperator*{\argmin}{arg\,min}
\DeclareMathOperator*{\dist}{d}
\newcommand{\interior}[1]{\operatorname{int}(#1)}

\newcommand{\real}{\mathbb{R}}
\newcommand{\realpos}{\mathbb{R}_{>0}}
\newcommand{\Dc}{\mathcal{D}}
\newcommand{\Zc}{\mathcal{Z}}
\newcommand{\Cc}{\mathcal{C}}
\newcommand{\Xc}{\mathcal{X}}
\newcommand{\Nc}{\mathcal{N}}
\newcommand{\dt}{\frac{d}{dt}}
\renewcommand{\theenumi}{(\roman{enumi})}

\allowdisplaybreaks

\begin{document}
\title{\LARGE \bf Continuity and Boundedness of Minimum-Norm CBF-Safe
  Controllers}

\author{Mohammed Alyaseen \quad Nikolay Atanasov \quad Jorge Cort\'es
  \thanks{M. Alyaseen, N. Atanasov, and J. Cort\'es are with the
    Contextual Robotics Institute, UC San Diego,
    \texttt{\{malyasee,natanasov,cortes\}@ucsd.edu}. M. Alyaseen
    is also affiliated with Kuwait University as a holder of a scholarship.} %
}

\maketitle

\begin{abstract}
  The existence of a Control Barrier Function (CBF) for a
  control-affine system provides a powerful design tool to ensure
  safety. Any controller that satisfies the CBF condition and ensures
  that the trajectories of the closed-loop system are well defined
  makes the zero superlevel set forward invariant.  Such a controller
  is referred to as \emph{safe}.  This paper studies the regularity
  properties of the minimum-norm safe controller as a stepping stone
  towards the design of general continuous safe feedback controllers.
  We characterize the set of points where the minimum-norm safe
  controller is discontinuous and show that it depends solely on the
  safe set and not on the particular CBF that describes it.  Our
  analysis of the controller behavior as we approach a point of
  discontinuity allows us to identify sufficient conditions to ensure
  it grows unbounded or it remains bounded.  Examples illustrate our
  results, providing insight into the conditions that lead to
  (un)bounded discontinuous minimum-norm controllers.
\end{abstract}

\section{Introduction}

Safety-critical control for dynamical systems is an active area of
research with applications to multiple domains such as transportation,
autonomy, power systems, robotics, and manipulation.  The notion of
Control Barrier Function (CBF) has revealed to be a particularly
useful tool as it provides a mathematically precise formulation of the
range of design choices available to keep a desired set safe.  This
has spurred a flurry of activity aimed at synthesizing safe
controllers as solutions to optimization-based formulations whose cost
functions may encode energy considerations, minimal deviation from
prescribed controllers, or other performance goals.  A critical aspect
in this endeavor is ensuring that safe controllers enjoy appropriate
regularity (boundedness, continuity, Lipschitzness, smoothness)
properties for ease of implementation and to ensure well-posedness of
the resulting closed-loop system.  Motivated by these observations,
this work studies the continuity properties of the minimum-norm safe
controller and analyzes conditions under which the existence of a
bounded safe controller is guaranteed.

\emph{Literature Review:} The notion of CBF builds on Nagumo's
theorem~\cite{MN:42}, which establishes the invariance of a set with
respect to trajectories of an autonomous system given suitable
transversality conditions are satisfied on the boundary of that
set. The extension to control systems introduced in~\cite{PW-FA:07}
enforces a strict Nagumo-like condition to hold on the whole set to be
made invariant. This condition was relaxed in~\cite{XX-PT-JWG-ADA:15}
to arrive at the concept of CBF used here.
The use of CBFs to enforce safety as forward set invariance has since
expanded to many domains (we refer
to~\cite{ADA-SC-ME-GN-KS-PT:19,WX-CGC-CB:23} for a comprehensive
overview).

Particularly useful is the fact that, if a CBF-certified safe
controller is Lipschitz, then the closed-loop system is well posed and
the superlevel set of the CBF is forward invariant.  It is common to
synthesize such controllers via optimization formulations which are
examples of parametric optimization problems, with the optimization
variable being the control signal and the parameter being the
state. The resulting controller is well defined but is generally not
guaranteed to be continuous, let alone Lipschitz.  If the controller
is discontinuous, then it might become unbounded even if the safe set
is compact, violating hard limits imposed by hardware constraints or
energy considerations. This has motivated the study in the literature
of various sufficient conditions to ensure Lipschitzness or continuity
of optimization-based controllers.  One set of
conditions~\cite{XX-PT-JWG-ADA:15} relies on assuming uniform relative
degree 1 of the CBF with respect to the dynamical system. Another
condition~\cite{RK-ADA-SC:21} asks that the properties defining the
CBF are satisfied on an open set containing the safe set. Other
works~\cite{BJM-MJP-ADA:15} derive continuity-ensuring conditions
resorting to the classical parametric optimization
literature~\cite{GS:18}, of which the optimization-based controller
synthesis problem is a special case.
In parametric optimization, the work~\cite{WWH:73} proves the
continuity of the optimizer under continuity properties of the
point-to-set map defined by the constraints.  Other works derive
continuity results under different types of constraint qualification
conditions, including linear independence~\cite{AVF-YI:90} and
Mangasarian-Fromovitz~\cite{SMR:82}. The work~\cite{BJM-MJP-ADA:15}
builds on this body of work to relax linear independence qualification
for the special case of a convex linearly constrained quadratic
parametric program. Our exposition here unifies these conditions under
a common framework and provides a generalization, ensuring continuity
of the min-norm safe controller under weaker conditions. We also
analyze the boundedness of the controller when the conditions are not
met and discontinuity arises.  Finally, because of the connection with
bounded control, relevant to the present work are methods for
constructing CBFs under limited control
authority~\cite{WSC-DVD:20,DRA-DP:21,AC:21a} and the combination of
CBFs with Hamilton-Jacobi reachability analysis to consider the impact
of control bounds on the computation of safe
sets~\cite{JJC-DL-KS-CJT-SLH:21}.


\emph{Statement of Contributions:} Given a CBF for a control-affine
system, we study the boundedness properties of the associated
minimum-norm safe controller. Apart from its intrinsic interest, the
focus on this controller is justified by the fact that if it is not
bounded, then no safe controller is.  We start by explaining the
limitations of the state of the art to guarantee the boundedness of
safe controllers and illustrating them in two examples.  Our first
contribution is a rigorous characterization of the points of
discontinuity of the minimum-norm safe controller.  As a byproduct,
this result allows us to generalize the known conditions for ensuring
continuity.  We show that the points of discontinuity are fully
determined by the safe set and are independent of the specific choice
of the CBF or the sensitivity to the violation of the CBF condition.
These results set the basis for our second contribution, which is the
identification of tight conditions to ensure the (un)boundedness of
the minimum-norm controller when approaching a point of discontinuity.
We revisit the two examples in light of the technical discussion to
explain the observed behavior of the minimum-norm controller.  Our
results are applicable to more general formulations of safety filters
beyond the minimum-norm controller and have important implications for
the synthesis of safe feedback controllers subject to hard constraints
on the control effort.


\emph{Notation:} The closure, interior, and boundary of a set $\Xc$
are denoted by $\bar {\Xc}$, $\interior{\Xc}$, and $\partial \Xc$,
respectively.  Given $s : \Xc \subseteq \real^n \to \real$, $s \in C$
denotes that $s$ is continuous and $s \in C^n$ denotes that $s$ has a
continuous $n^{\text{th}}$ derivative. The gradient of $s \in C^1$ is
denoted by $\nabla s$ and written as a row vector. A function $s$ is
locally Lipschitz at $x$ with respect to $\Xc$ if there exists a
neighborhood $\Nc$ and a constant $L \in \real$ such that
$\|s(x_1) - s(x_2)\| \leq L \|x_2 - x_1\|$, for all
$x_1, x_2 \in \Nc \cap \Xc$. A function $s$ is locally Lipschitz on
$\Xc'$ if it is locally Lipschitz at $x$ with respect to $\Xc'$, for
all $x \in \Xc'$. A function $\alpha: (-a,b) \to \real$ is an extended
class-$\kappa$ function if it is strictly increasing and
$\alpha(0) = 0$.

\section{Problem Statement}
We consider a non-linear control affine system over an open set
$\Xc \subseteq \real^n$
%
\begin{align}
  \dot x = f(x) + G(x)u \label{eq:mimoSystem},
\end{align}
where $x\in \Xc$ and $u\in \real^m$. Here,
$f:\Xc \to \real^n$ and the column components
$g_i:\Xc \to \real^{n}$, $i \in \until{m}$ of $G$ are
locally Lipschitz on $\Xc$. Safety of the system can be
certified through the following notion.

\begin{definition}[Control Barrier
  Function~\cite{ADA-SC-ME-GN-KS-PT:19}]\label{def:CBFs}
  Let $h:\Xc \to \real$ be $C^1$ and define its superlevel
  set
  $\Cc \triangleq \{x\in \real^n \;|\; h(x) \geq
  0\}\subseteq \Xc$.  The function $h$ is a CBF if
  $\nabla h(x) \neq 0 $ for all $x \in \partial \Cc $ and there
  exists a set $\Dc \subseteq \Xc$ such that
  $\Cc  \subseteq \Dc$ and for all $x \in \Dc$,
  there exists $u \in \real^m$,
  \begin{align}
    \nabla h(x)f(x) + \alpha(h(x)) + \nabla h(x)G(x)u \geq
    0 .\label{eq:cbfConditionD} 
  \end{align}
  where $\alpha$ is an extended class-$\kappa$ function. 
\end{definition}

If $h$ admits an open set $\mathcal{D}$ satisfying the above
definition, then we refer to it as a \emph{strong CBF}, otherwise we
call it a \emph{weak CBF}. For each $x \in \Dc$, we denote by
$K_{\text{cbf}}(x)$ the set of input values $u$
satisfying~\eqref{eq:cbfConditionD} which, by
Definition~\ref{def:CBFs}, is nonempty.

The central result~\cite[Theorem 2]{ADA-SC-ME-GN-KS-PT:19} of
CBF-based safety is that, if there exists a Lipschitz feedback
controller $\bar u: \real^n \to \real^m$ satisfying
$\bar u(x) \in K_{\text{cbf}}(x)$ in $\Dc$, then the set $\Cc $ is
forward invariant with respect to the trajectories of the closed-loop
system~\eqref{eq:mimoSystem} under $u = \bar u(x)$. One particular
choice of controller that satisfies the CBF
condition~\eqref{eq:cbfConditionD} by construction is the so-called
min-norm safe feedback controller
$u^*(x) \triangleq \argmin_{u\in K_{\text{cbf}}(x)} \|u\|^2$.  In
general, this controller is not necessarily Lipschitz. In fact, it
might not even be bounded. This motivates our problem statement.

\begin{problem}\label{pr:boundednessProb}
  {\rm Let $h$ be a CBF with a compact superlevel set $\Cc $.
    Determine the states in $\Cc $ where the min-norm safe
    feedback controller $u^*$ is discontinuous and find conditions
    under which it is bounded/unbounded as the discontinuous states
    are approached.} \oprocend
\end{problem}

Our focus on establishing boundedness when continuity of the min-norm
controller fails is motivated by three reasons. First, proving that
the min-norm controller is unbounded shows that no safe bounded
controller exists. This would also mean that there does not exit a
continuous safe feedback controller. Second, if the min-norm is
discontinuous but bounded, then there is room for finding a safe
continuous controller. Finally, our investigation provides grounds for
exploring whether the use of discontinuous controllers to ensure
control-invariance for safety is applicable to a larger class of
scenarios.

We end this section by noting that our results are directly applicable
to safety filters based on quadratic programming (QP). In fact, any
controller $u$ that minimizes a cost function $\|u - u_{\nom}(x)\|^2$
subject to \eqref{eq:cbfConditionD}, where $u_{\nom}$ is a predefined
nominal controller, can be interpreted as a min-norm controller after
the change of variables~$u' = u - u_{\nom}$.

\section{Continuity of the Min-Norm Safe Controller: Limitations of
  the State of the Art}\label{sec:Egs}

This section reviews known conditions in the literature that ensure
the min-norm controller $u^*$ is continuous and thus bounded in a
compact set $\Cc $, and illustrates its limitations in a couple
of simple examples.
Considering the CBF
condition~\eqref{eq:cbfConditionD}, notice that if
$\nabla h(x)f(x) + \alpha(h(x)) \geq 0$, then $u = 0$
validates~\eqref{eq:cbfConditionD}. For such points, the min-norm
controller $u^*(x)=0$. On the other hand, when
$\nabla h(x)f(x) + \alpha(h(x)) < 0$, a non-zero control is needed to
ensure~\eqref{eq:cbfConditionD}. We thus split $\Dc$ into the
two sets
\begin{subequations}\label{eq:min-norm-piecewise}
  \begin{align}
    \Dc_+ &\triangleq \{x \in \Dc \;\;|\;\; \nabla h(x)f(x) +
        \alpha(h(x)) \geq 0\}, \label{eq:S}
    \\ 
    \Dc_- &\triangleq \{x \in \Dc \;\;|\;\; \nabla h(x)f(x) +
        \alpha(h(x)) < 0\}.\label{eq:U}
\end{align}  
\end{subequations}
Notice that $u^*$ is defined as the optimizer of a quadratic program
with one linear constraint. Such programs have a unique solution,
cf.~\cite[8.1.1]{SB-LV:09}, with the closed-form formula
{\small
\begin{align}\label{eq:minnorm}
  u^*(x) \!=\! 
  \begin{cases}
    0 , & x \in \Dc_+
    \\
    -\frac{\nabla h(x)f(x) + \alpha(h(x))}{\|\nabla
      h(x)G(x)\|^2}(\nabla h(x)G(x))^T, &
     x \in \Dc_-.
  \end{cases}
\end{align}
}
This expression is well defined on $\Dc$
since~\eqref{eq:cbfConditionD} implies that, if $\bar x \in \Dc_-$,
then $\|\nabla h(\bar x)G(\bar x)\| \neq 0$.

\begin{lemma}[Strong CBF Implies Continuous Min-Norm
  Controller~{\cite[Thm.~5]{RK-ADA-SC:21}}]\label{lem:strongImpliesBounded}
  Let $h$ be a strong CBF with a compact superlevel set $\Cc $.
  Then $u^*$ is continuous on $\Cc $.
\end{lemma}



According to~\cite[Thm. 8]{XX-PT-JWG-ADA:15}, $u^*$ is locally
Lipschitz if the CBF $h$ has relative degree 1, that is, for all
$x \in \Dc$, $\|\nabla h(x)G(x)\| \neq 0$.
%
%
The next result is a generalization of this fact.

\begin{lemma}[Generalization of Relative Degree 1 CBF Implies
  Continuous Min-Norm Controller]\label{lem:LghNeq0}
  Let $h$ be a CBF with compact superlevel set $\Cc $. If for
  all $x \in \partial \Cc $, $\|\nabla h(x)G(x)\| = 0$ implies
  $\nabla h(x)f(x)> 0$, then $u^*$ is locally Lipschitz on
  $\Cc $.
\end{lemma}
We postpone the proof of Lemma~\ref{lem:LghNeq0} as it is a corollary of
Lemma~\ref{lem:contOfMinnorm} below.

\begin{remark}[Assumption of uniform relative degree is
  limiting]\label{re:limiting}%
  {\rm The assumption of uniform relative degree of the CBF,
    cf.~\cite[Thm. 8]{XX-PT-JWG-ADA:15}, has also been exploited for
    higher-order relative degree CBFs, cf.~\cite{QN-KS:16}. However,
    this assumption fails for the following two general
    cases: 
    \begin{enumerate}
    \item Let $h$ be a continuously differentiable CBF with compact
      superlevel set $\Cc$. For such $h$, there always exists $y \in
      \interior{\Cc}$ 
      where $\|\nabla h(y)G(y) \| = 0$. To see that, note that by
      continuity of $h$ and compactness of its superlevel set, $h$ has a
      maximum value at some state $y \in {\Cc}$
      \cite[Thm. 4.16]{WR:76}. Recalling that $h(x) = 0$ at
      $\partial \Cc$ and $h(x) > 0$ in $\interior{\Cc}$, we deduce that
      $y \in \interior{\Cc}$. By differentiability and first-order
      optimality~\cite[4.2.3]{SB-LV:09}, $\nabla h(y) = 0$ and, hence,
      $\|\nabla h(y)G(y)\| = 0$.
    \item Consider the $n$-dimensional linear system $(A,B)$, where
      $B$ does not have full row rank. Let $h$ be a continuously
      differentiable CBF with compact convex superlevel set
      $\Cc$. Then, there always exists $y \in \partial \Cc$ where
      $\|\nabla h(y)G(y)\| = \|\nabla h(y)B\| = 0$. To see this, note
      that since $B$ is not full row rank, there is a unit vector
      $v \in \real^n$ such that $\|v^TB\| = 0$. By the surjectivity of
      the Gauss map \footnote{The Gauss map assigns points on the
        manifold $\partial \Cc$ to the unit sphere embedded in
        $\real^n$ such that the image of any point in $\partial \Cc$
        is the unit vector normal to $\partial \Cc$ at that point.}
      on the compact smooth surface $\partial
      \Cc$~\cite[Thm. A]{RL:75}, there is a point $y \in \partial \Cc$
      at which the unit normal vector to $\partial \Cc$ is
      $v$. By~\cite[Thm. 3.15]{NB:23}, $\nabla h(y)$ is normal to
      $\partial \Cc$ at $y$ and thus parallel to $v$. Hence,
      $\|\nabla h(y)B\| = 0$.  \oprocend
    \end{enumerate}
  }
\end{remark}

From the continuity of the
min-controller on $\Cc $ ensured by either
Lemmas~\ref{lem:strongImpliesBounded} or~\ref{lem:LghNeq0}, it follows
from standard results in analysis, cf.~\cite[Thm. 5.15]{WR:76}, that
$u^*$ is bounded if $\Cc $ is compact. As we will show later, the
conditions of Lemmas~\ref{lem:strongImpliesBounded}
and~\ref{lem:LghNeq0} are not totally independent: rather, if the
condition of Lemma~\ref{lem:strongImpliesBounded} is not met, i.e.,
$h$ is weak, then the condition of Lemma~\ref{lem:LghNeq0} is not met
either.


CBFs that do not meet the conditions of these results are easy to
encounter and arise in practice
in contexts as simple as the problem of confining a double integrator
to a circle centered at the origin. We next present two examples that
do not satisfy the assumptions and generate discontinuous min-norm
controllers: one being bounded and the other one unbounded.

\begin{example}[Weak CBF with Bounded Min-Norm
  Controller]\label{ex:doubleIntegrator1}\rm{
    Consider the double-integrator dynamics on $\real^2$ defined by
    $f(x) = (x_2,0)$ and $G(x) = (0,1)$. The function
    $h(x) = 1 - x_1^2 - x_2^2$
    is a CBF with any extended class-$\kappa$ function $\alpha$.
    Notice further that $h$ is a weak CBF. To see this, let
    $\bar x = (1+\epsilon, 0)$ with any arbitrarily small
    $\epsilon > 0$. Since $\bar x \notin \Cc $, we have
    $\nabla h(\bar x)f(\bar x) + \alpha (h(\bar x)) + \nabla h(\bar
    x)G(\bar x)u = \alpha (h(\bar x)) < 0$, and therefore condition
    \eqref{eq:cbfConditionD} cannot be satisfied at $\bar{x}$.
    Therefore, $h$ does not admit an open set $\Dc$ satisfying
    Definition~\ref{def:CBFs}.  In addition, the condition of
    Lemma~\ref{lem:LghNeq0} is not satisfied at the boundary point
    $(1,0)$. Consider now the norm of the min-norm safe
    controller~\eqref{eq:minnorm} defined on
    $ \Dc = \Cc $,
    \begin{align*}
      |u_1^*(x)|
      &= 
        \begin{cases}
          0, & x \in \Dc_+, \\
          \frac{2x_1x_2 - \alpha(h(x))}{2x_2}, & x \in \Dc_- .
        \end{cases}
    \end{align*}
    Note that $u_1^*$ is continuous on
    $\Cc  \setminus \{(\pm 1, 0)\}$. However, choosing $\alpha(r) = r$, we have that
    $\limsup_{x \to (1,0), x \in \Dc_-}|u_1^*(x)|$ and
    $\lim_{x \to (1,0), x \in \Dc_+}|u_1^*(x)| = 0$. Thus, although
    discontinuous at $(1,0)$, $u_1^*$ is bounded at this point,
    cf. top plot in Figure~\ref{fig:eg12}.} \oprocend
\end{example}

Example~\ref{ex:doubleIntegrator1} shows that the min-norm safe
controller might be bounded even if the CBF does not satisfy the continuity
conditions in the literature.  The next example shows this fact is not
generic.

\begin{example}[Weak CBF with Unbounded Min-Norm
  Controller]\label{ex:unboundedEg}
  {\rm Consider the dynamics $f(x) = (x_2,0)$ and
    $G(x) = (0,x_2^2)$. With the same reasoning as in Example~\ref{ex:doubleIntegrator1}, $h(x) = 1 - x_1^2 - x_2^2$ is a weak
    CBF that does not satisfy the requirement of
    Lemma~\ref{lem:LghNeq0}. The norm of the min-norm safe controller
    is:
    \[
      |u_2^*(x)| = \begin{cases} 0, & x \in \Dc_+, \\
        \frac{2x_1x_2 - \alpha(h(x))}{2x_2^3}, & x \in \Dc_-.
      \end{cases}
    \]
    Observe that $u_2^*$ is continuous on
    $\Cc  \setminus \{(\pm 1, 0)\}$. However, with the choice $\alpha(r) = r$,
    $\limsup_{x \to (1,0), x \in \Dc_-}|u_2^*(x)| = \infty$ and
    $\lim_{x \to (1,0), x \in \Dc_+}|u_2^*(x)| = 0$. Thus, $u^*_2$ is
    neither continuous nor bounded on $\Cc $, cf. bottom plot
    in Figure~\ref{fig:eg12}.} \oprocend
\end{example}  

\begin{figure}
  \centering
  \includegraphics[width=0.35\textwidth]{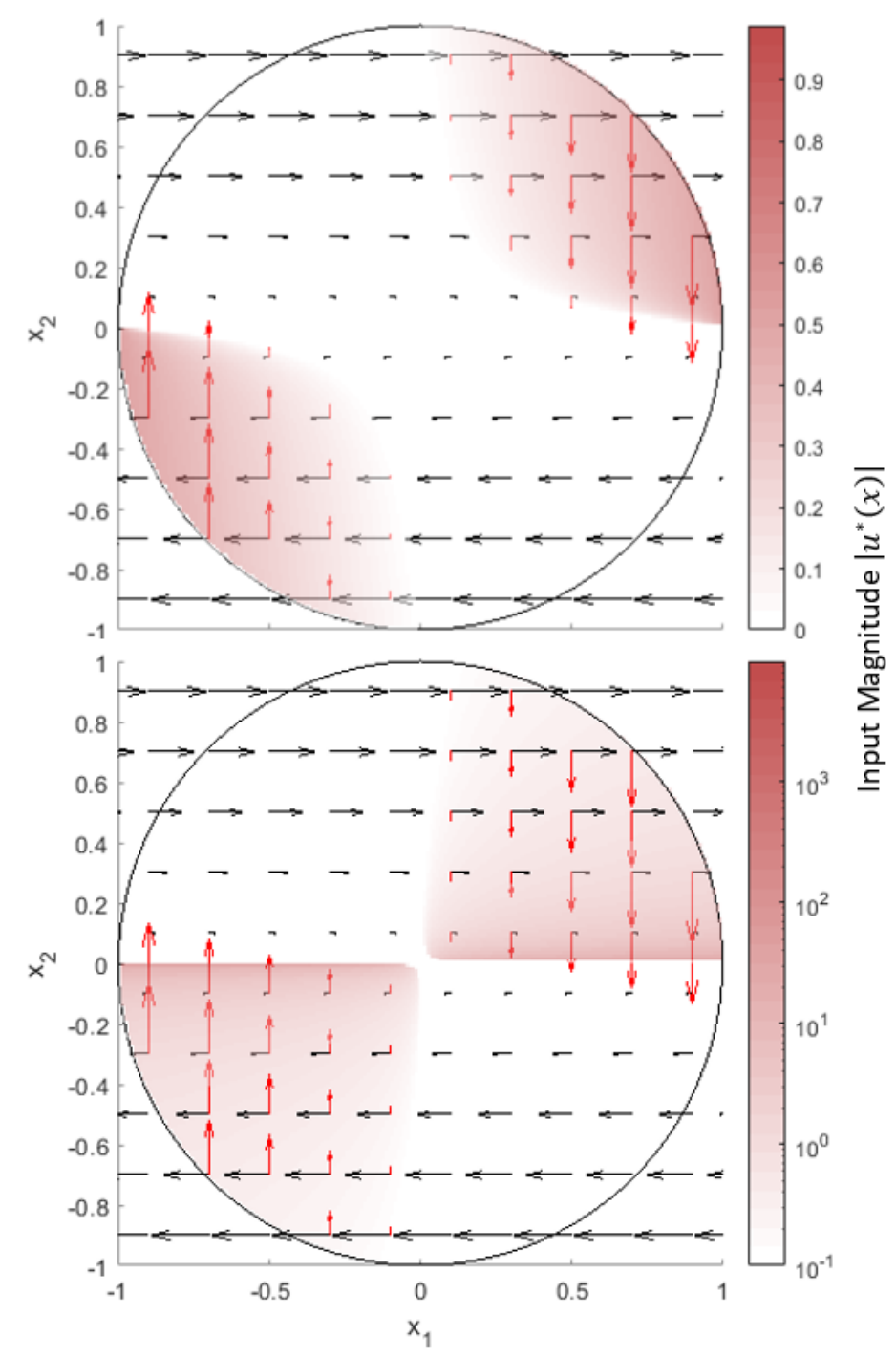}
  \caption{Illustration of boundedness of the min-norm safe
    controller.  Top (resp. bottom) plot corresponds to
    Example~\ref{ex:doubleIntegrator1} (resp.,
    Example~\ref{ex:unboundedEg}). In each case, the unit circle is
    the superlevel set of the weak CBF $h$, black arrows show the
    vector field $f(x)$, red arrows show $G(x)u^*(x)$, and the color
    map shows the magnitude of the input $u^*$.}
  \label{fig:eg12}
\end{figure}

\section{Points of Discontinuity of The Min-Norm Safe
  Controller}\label{sec:discontinuity}

Here we characterize the points of (dis)continuity of the min-norm
controller~$u^*$ in $\Cc $. This is motivated by the fact that
if $u^*$ goes unbounded when approaching a point in $\Cc $,
then it is discontinuous at it. Therefore, the results of this section
are a stepping stone towards the identification of conditions for
(un)boundedness of~$u^*$.

%
%

%

\begin{lemma}[Points of discontinuity of $u^*$ in
  $\Cc $]\label{lem:contOfMinnorm}
  Let $h$ be a CBF for a system \eqref{eq:mimoSystem} with a Lipschitz gradient and an
  associated Lipschitz class-$\kappa$ function $\alpha$, and let $u^*$ be the
  min-norm controller given by \eqref{eq:minnorm}. Define
  $ \Zc_{h,\alpha} \triangleq \{x\in \Cc \; |\; \nabla h(x)f(x) + \alpha(h(x)) = 0 = \|\nabla h(x)G(x)\|
  \}$. Then, $u^*$ is locally Lipschitz on
  $\Cc  \setminus \Zc_{h,\alpha}$.
\end{lemma}
\begin{proof}
  The proof is an extension of the proof of~\cite[Thm. 8]{XX-PT-JWG-ADA:15}. Note that Since $h$ is a CBF,~\eqref{eq:cbfConditionD} is satisfied for $\Dc = \Cc $
  and therefore $\| \nabla h(x)G(x)\| \neq 0$, for all $x \in \Dc_-$.
  Thus, on $\Dc_-$, $u^*$ is a quotient with a non-zero Lipschitz
  denominator and a Lipschitz numerator. Hence, both expressions in
  the piecewise definition of $u^*$ in~\eqref{eq:minnorm} are locally
  Lipschitz on their respective domains $\Dc_+$ and $\Dc_-$. It
  remains to prove that $u^*$ is locally Lipschitz with respect to
  $\Cc $ at all the points in the boundary between $\Dc_+$ and
  $\Dc_-$ that are not in $\Zc_{h,\alpha}$. For a point $x$ in
  the boundary between $\Dc_+$ and $\Dc_-$,
  $\nabla h(x)f(x) + \alpha(h(x)) = 0$. If at such a point
  $\|\nabla h(x)G(x)\| \neq 0$ (i.e.,
  $x \notin \Zc_{h,\alpha}$), then there is a neighborhood
  $\Nc$ of $x$ such that $\|\nabla h(y)G(y)\| \neq 0$ for all
  $y \in \Nc$. Thus
  $u^*(x) = \omega(\frac{\nabla h(x)f(x) + \alpha(h(x))}{\|\nabla
    h(x)G(x)\|})(\nabla h(x)G(x))^T$ for $x \in \Nc$, where
  \begin{align*}
    \omega(r) = \begin{cases}
      0, & r \geq 0 ,\\
      -r, & r < 0 ,
    \end{cases}
  \end{align*}
  which is locally Lipschitz on $\real$.  That $u^*$ is locally
  Lipschitz at $x$ follows from the facts that the composition and
  product of locally Lipschitz functions is locally Lipschitz, and the
  quotient of locally Lipschitz functions is locally Lipschitz
  provided that the denominator is not zero.
  %
\end{proof}

Lemma~\ref{lem:contOfMinnorm} can be seen as an extension of previous
results, cf.~\cite[Thm. 8]{XX-PT-JWG-ADA:15}, establishing local
Lipschitzness of $u^*$ by assuming uniform relative degree 1 of
$h$. If this is the case, then $\Zc_{h,\alpha}$ is empty and
thus $u^*$ is locally Lipschitz on~$\Cc $.
Given the dependency of $\Zc_{h,\alpha}$ on $h$ and $\alpha$,
one might consider the possibility that a suitable choice of these
functions might eliminate the potential points of discontinuity. The
following results rule this out.

\begin{lemma}[Discontinuity Points Are Independent of
  $\alpha$]\label{lem:ZinBoundary}
  Let $h$ be a CBF. Then there exists an extended class-$\kappa$
  function $\alpha$ that validates the CBF condition
  \eqref{eq:cbfConditionD} and such that
  $\Zc_{h,\alpha} \subseteq \partial \Cc$. Moreover, let
  $\alpha_1$ and $\alpha_2$ be two extended class-$\kappa$ functions
  that validate the CBF definition for $h$. Then
  $\Zc_{h,\alpha_1} \cap \partial \Cc  = \mathcal
  Z_{h,\alpha_2} \cap \partial \Cc $.
\end{lemma}
\begin{proof}
  We prove that if $\alpha$ validates Definition~\ref{def:CBFs} for
  $h$, then any class-$\kappa$ function $\bar \alpha$ that satisfies
  $\bar \alpha(r)>\alpha(r)$ for all $r>0$ validates
  Definition~\ref{def:CBFs} for $h$ and gives
  $\Zc_{h,\bar {\alpha}} \cap \interior{\Cc } = \emptyset$.  That
  $\bar {\alpha}$ validates the CBF condition \eqref{eq:cbfConditionD}
  is immediate. Now let $\bar x \in \interior{\Cc }$ be such that
  $\nabla h(\bar x)f(\bar x) + \bar {\alpha}(h(\bar x)) = 0$. We show
  that $\|\nabla h(\bar x) G(\bar x) \|\neq 0$ and thus
  $\bar x \notin \Zc_{\bar {\alpha},h}$. Since
  $\bar \alpha(r)>\alpha(r)$ for $r>0$,
  $\nabla h(\bar x)f(\bar x) + {\alpha}(h(\bar x)) < 0$ because
  $h(\bar x) > 0$ as $\bar x\in \interior{\Cc }$. But $\alpha$
  validates condition \eqref{eq:cbfConditionD} and thus
  $\|\nabla h(\bar x) G(\bar x) \|\neq 0$. The proof of the last claim
  in the statement is immediate from the fact that
  $\alpha_1(h(x)) = \alpha_2(h(x)) = 0$ on $\partial \Cc $.
\end{proof}

If we thus define
\begin{align}
  \Zc_h \triangleq \{x\in \partial \Cc  \; |
  \; \nabla h(x)f(x) = \|\nabla h(x)G(x)\| = 0\}, \label{eq:Z}
\end{align}
then Lemmas~\ref{lem:contOfMinnorm} and~\ref{lem:ZinBoundary} justify
stating that $u^*$ is continuous on
$\Cc  \setminus \Zc_h$. This shows that $u^*$ is
continuous on $\interior{\Cc }$ and that the possible points of
discontinuity are independent of the choice of~$\alpha$.

\begin{lemma}[Discontinuity Points Are Independent of
  $h$]\label{lem:ZindependentOnH}
  Let $h_1$, $h_2 \in C^1$ be CBFs with the same superlevel set
  $\Cc $. Then, $\Zc_{h_1} = \Zc_{h_2}$.
\end{lemma}
\begin{proof}
  By Definition~\ref{def:CBFs}, $\nabla h_i(x) \neq 0$,
  $i \in \{1,2\}$ on $\partial \Cc $. By~\cite[Thm. 5.1]{MS:95},
  both $h_1 = 0$ and $h_2 = 0$ define the same differentiable manifold
  $\partial \Cc $ of dimension $n-1$ embedded in $\real^n$. By~\cite[Thm. 3.15]{NB:23}, the tangent space $T_x$ of this
  manifold at a point $x$ is given by
  $T_x = \text{kernel}(\nabla h_1(x)) = \text{kernel}(\nabla
  h_2(x))$. Thus $\nabla h_1(x)$ and $\nabla h_2(x)$ are parallel, and
  the result follows using the definition of~$\Zc_h$.
\end{proof}

Lemma~\ref{lem:ZindependentOnH} shows that $\Zc_h$ is
associated to the set $\Cc $ and is independent of the CBF that
has this set as its superlevel set. We thus write $\Zc$ to
denote $\Zc_h$ without loss of generality.

Lemma~\ref{lem:LghNeq0} can now be readily proved: in fact, the
hypotheses there imply that $\Zc$ is empty, and therefore, by
Lemma~\ref{lem:contOfMinnorm}, $u^*$ is continuous on $\mathcal
C$. Now that it is proved that the non-emptiness of the set
$\Zc$ implies potential discontinuity; one might then hope that
boundedness of $u^*$ can be established for a weak CBF $h$ by ensuring
that $\Zc$ is empty. The next result shows that the latter is
never the case.

%
%
\begin{lemma}[Weak CBF Implies Possible
  Discontinuity]\label{lem:Znonempty}
  If $h$ is a weak CBF, then $\Zc$ is nonempty.
\end{lemma}
\begin{proof}
  Define the sequence of sets
  $\Dc_{n} \triangleq \{x\in \real^n \;|\;
  \dist(x,\Cc ) < 1/n\}$,
  where $\dist(x,\Cc )$ is the distance function from $x$ to set
  $\Cc $, which is continuous,
  cf.~\cite[Thm. 3.1]{EG-DJ:85}. Note that
  $\Cc  \subset \Dc_n$ and $\Dc_{n}$ is open for
  all $n \in \mathbb N$. Since $h$ is a weak CBF, for each
  $ n \in \mathbb N$, there exists
  $x_n \in \Dc_n \setminus \Cc $ such that for all
  $u \in \real^m$ and all class-$\kappa$ functions $\alpha$,
  $\nabla h(x_n)f(x_n) + \alpha(h(x_n)) + \nabla h(x_n) G(x_n)u < 0$.
  This implies that necessarily $\|\nabla h(x_n) G(x_n)\| = 0$ and
  $\nabla h(x_n)f(x_n) + \alpha(h(x_n)) < 0$.
  Consider the sequence $\{x_n\}$. Since $\Cc $ is compact, the
  closure of $\Dc_1$, namely $\bar {\Dc}_1$, is
  compact. Since $\{x_n\} \subseteq \bar {\Dc}_1$, there
  exists, cf.~\cite[Thm. 3.6]{WR:76}, a convergent subsequence of
  $\{x_n\}$, denoted $\{y_n\}$, whose limit is $\bar y$. By the
  definition of $\{y_n\}$, we have $\dist(y_n,\Cc ) \to 0$, and by
  continuity, $\dist(\bar y, \Cc ) = 0$, and so
  $\bar y \in \Cc $. Since $h(y_n) < 0$ for all $n$, it follows
  that $h(\bar y) \leq 0$, and therefore it must be that
  $h(\bar y) = 0$, i.e., $\bar{y} \in \partial \Cc $.
  Continuity and the fact that $\|\nabla h(y_n) G(y_n)\| = 0$ for all
  $n$ implies $\|\nabla h(\bar y) G(\bar y)\| = 0$. Similarly,
  continuity and the fact that
  $\nabla h(y_n)f(y_n) + \alpha(h(y_n)) < 0$ implies that
  $\nabla h(\bar y)f(\bar y) + \alpha(h(\bar y)) = \nabla h(\bar
  y)f(\bar y) \leq 0$.  Since $h$ is a CBF and
  $\bar y \in \Cc $, we have
  $\nabla h(\bar y)f(\bar y) + \alpha(h(\bar y)) = \nabla h(\bar
  y)f(\bar y) \geq 0$. Therefore $\nabla h(\bar y)f(\bar y) = 0$ and
  thus, $\bar y \in \Zc$, implying $\Zc \neq \emptyset$.
\end{proof}

Lemma~\ref{lem:Znonempty} provides an important connection between the
conditions for continuity presented in Section~\ref{sec:Egs}. In fact,
if the CBF is not strong, but weak (i.e., the condition of
Lemma~\ref{lem:strongImpliesBounded} is not met), then
Lemma~\ref{lem:Znonempty} implies that the condition of
Lemma~\ref{lem:LghNeq0} is not satisfied either.

\section{(Un)Boundedness Conditions For The Min-Norm Safe Controller}

This section identifies conditions to determine when the min-norm
controller is bounded.  For a compact safe set $\Cc $, the controller
can go unbounded only if approaching a state at which it is
discontinuous (see e.g., Example~\ref{ex:unboundedEg} for an
illustration). From the exposition in Section~\ref{sec:discontinuity},
we know that the points of discontinuity of the min-norm controller
are contained in $\Zc$, cf.~\eqref{eq:Z}.  The following result
provides computable sufficient conditions for (un)boundedness when
approaching a point in~$\Zc$.

\begin{theorem}[(Un)Boundedness Conditions of Min-Norm
  Controller]\label{thm:mimoBoundedness}
  Let $h \in C^2$ be a CBF with compact superlevel set $\Cc $ and an
  associated $\alpha$ that is differentiable at $0$.  Assume $f$ and
  $G$ are differentiable at $\bar x \in \Zc$ and let $H_h(\bar x)$,
  $J_f(\bar x)$, and $J_{g_i}(\bar x)$ denote the Hessian of $h$ and
  the Jacobians of $f$ and $g_i$, respectively. Consider the linear
  equation
  \begin{align}\label{eq:mimoLinearTest}
    Av
    &=
      \begin{bmatrix}
        c_1
        \\
        c_2
        \\
        \bold{0}
      \end{bmatrix},
  \end{align}
  with $v \in \real^n$, $c_1,c_2 \in \real$. Here, $\bold{0}$ is the
  zero vector in $\real^m$,
  $A \triangleq \begin{bmatrix} \nabla h(\bar x)^T & \beta_f(\bar x) &
    \beta_{G}(\bar x)\end{bmatrix}^T$ and
  \begin{align*}
    \beta_f(x)
    &\triangleq H_h(x)f(x) + (J_{f}^T(x) + \alpha
      '(h(x))I_n)\nabla h(x)^T \in \real^n,
    \\
    \beta_{g_i}(x)
    &\triangleq H_h(x)g_i(x) + J_{g_i}^T(x)\nabla h(x)^T
      \in \real^n,
    \\
    \beta_G(x)
    &\triangleq
      \begin{bmatrix}
        \beta_{g_1}( x)
        & \dots
        &
          \beta_{g_m}( x)
      \end{bmatrix} \in \real^{n \times m}. 
  \end{align*}
  Then, the following statements hold:
  \begin{enumerate}
  \item if \eqref{eq:mimoLinearTest} has a solution $v$ with
    $c_1 \geq 0$ and $c_2 < 0$, then $u^*$ is not bounded as
    $x \to \bar x$ in $\Cc $ from the direction of $v$, i.e.,
    $u^*(\bar x + vt)$ goes unbounded as $t \to 0^+$.
    
  \item if \eqref{eq:mimoLinearTest} does not have any non-trivial
    solution with $c_1 \geq 0$ and $c_2 \leq 0$, then $u^*$ is bounded
    as it approaches $\bar x$ from all possible directions in $\Cc $.
  \end{enumerate}
\end{theorem}
\begin{proof}
  The proof proceeds by examining the limit
  $\limsup_{t \to 0} \|u^*(\bar x + vt)\|$ for
  $v \in \real^n $. In doing so, we face the challenge that $u^*$ is
  given by a piecewise expression that is generally discontinuous
  at~$\bar x$. In addition, when computing the limit, one finds an
  indeterminate form of the type $0/0$. This leads us to the use of a
  particular form of L'H\^opital's rule~\cite{WR:76} that can handle
  the discontinuous piecewise expression and the presence of the
  $\limsup$.

  For brevity, we use $\bar x_t \triangleq \bar x + vt$,
  $h_G(t) \triangleq \nabla h(\bar x_t)G(\bar x_t)$,
  $N(t) \triangleq \nabla h(\bar x_t)f(\bar x_t) + \alpha(h(\bar
  x_t))$, and $D(t) \triangleq \|h_G(t)\|$. According to~\eqref{eq:U},
  $\|u^*(\bar x_t)\| = \frac{-N(t)}{D(t)}$ for $\bar x_t \in
  \Dc_-$. 

  (i) Let $v$ be a solution of \eqref{eq:mimoLinearTest} with
  $c_1 \geq 0$ and $c_2 < 0$. Because of the first row
  of \eqref{eq:mimoLinearTest}, we have that
  $\dt h(\bar x_t) = \nabla h(\bar x)v = c_1 \geq 0$. If
  $\nabla h(\bar x)v > 0$, then by continuity,
  $\nabla h(\bar x_t) > 0$ for small enough $t$. Thus
  by~\cite[Thm. 5.11]{WR:76}, $h(\bar x_t) > 0$, i.e.,
  $\bar x_t \in \Cc $, for small enough $t$. If
  $\nabla h(\bar x)v = 0$, then $v$ is tangential to $\mathcal
  C$. Hence $\bar x_t$ approaches $\bar x$ from within $\Cc $ or
  tangentially to it, meaning that $v$ is a valid direction of
  approach to consider.  The second row of~\eqref{eq:mimoLinearTest}
  ensures that $\dt N(t)|_{t = 0^+} = v^T\beta_f(\bar x) =c_2 < 0$,
  which again by~\cite[Thm. 5.11]{WR:76} proves that $N(t) < 0$, i.e.,
  $\bar x_t \in \Dc_-$ by \eqref{eq:U}, for sufficiently small $t$.
  Hence,
  $ \lim_{t \to 0^+}\|u^*(\bar x_t)\| = \lim_{t \to 0^+}
  \frac{-N(t)}{D(t)}$. Direct evaluation of this expression at $t = 0$
  (where $\bar x_t = \bar x$) yields an indeterminate form of the type
  $0/0$.
  We therefore resort to L'H\^opital's rule~\cite[Thm. 5.13]{WR:76},
  which requires the existence of the limit of the derivative of the
  numerator $-N(t)$ and denominator $D(t)$. For the numerator, we have
  already established $\lim_{t\to 0^+}\dt N(t) = c_2$.  As for
  the denominator, it is the norm of the differentiable function
  $h_G(t)$, and its derivative exists at $t$ where $h_G(t) \neq
  0$. But since $\bar x_t \in \Dc_-$ for small enough $t$, the CBF
  condition \eqref{eq:cbfConditionD} ensures that $h_G(t) \neq 0$ for
  sufficiently small $t$. Thus, the derivative of the denominator
  exists for sufficiently small $t>0$. A proof of the existence of the
  limit of this derivative
  $\lim_{t \to 0^+}\dt(D(t)) = \lim_{t \to 0^+} v^T
  \beta_G(\bar x_t) \frac{h_G(t)}{\|h_G(t)\|}$ follows. By
  H{\"o}lder's inequality,
  \begin{align*}
    \Big | \frac{v^T \beta_G(\bar x_t)  h_G(t)}{\|h_G(t)\|} \Big |
    \! \leq \! \|v^T \beta_G(\bar x_t) \|
    \frac{\|h_G(t)\|}{\|h_G(t)\|} \!=\! \| 
    v^T \beta_G(\bar x_t) \| .
  \end{align*}
  Hence, using the last $m$ rows of~\eqref{eq:mimoLinearTest}, the
  assumption of continuous differentiability, and the sandwich theorem
  for limits~\cite[Thm. 3.3.3]{HHS:03},
  $\lim_{t \to 0^+}\dt D(t) = 0$. By L'H\^opital,
  $\lim_{t \to 0^+}\|u^*(\bar x_t)\| = \lim_{t \to 0^+}
  \frac{-N(t)}{D(t)} = \lim_{t \to 0^+} \frac{-N'(t)}{D'(t)} =
  \infty$.

  (ii) We prove the contrapositive: assume there exists a vector $v$
  such that $\bar x_t = \bar x + vt$ approaches $\bar x$ from within
  $\Cc $ or tangent to it as $t \to 0^+$ and
  $\limsup_{t \to 0^+}\|u^*(\bar x_t)\| = \infty$, and let us show
  that then $v$ solves \eqref{eq:mimoLinearTest} with $c_1 \geq 0$ and
  $c_2 \leq 0$. 
  Note that $\nabla h(\bar x)v \geq 0$, since otherwise, under the
  theorem assumptions, for sufficiently small $t$,
  $\bar{x}_t \notin \Cc $, i.e., $\bar{x}_t$ would approach $\bar{x}$
  from outside $\Cc $, which is a contradiction.  This ensures the
  satisfaction of the first row
  in~\eqref{eq:mimoLinearTest}. Similarly, if
  $v^T\beta_f(\bar x) > 0$, then under the theorem assumptions, for
  sufficiently small $t$, $\bar x_t \in \Dc_+$ and thus
  $\limsup_{t \to 0^+}u^*(\bar x_t) = 0$, which contradicts
  $\limsup_{t \to 0^+}\|u^*(\bar x_t)\| = \infty$.  This ensures the
  satisfaction of the second row in
  \eqref{eq:mimoLinearTest}. According to Lemma~\ref{lem:halfThm1},
  there exists a sequence $\{\bar t_i\} \to 0^+$ with
  $\{\bar x_{\bar t_i}\} \subset \Dc_-$ such that
  $D'(\bar t_i) = v^T \beta_G(\bar x_{\bar t_i}) \frac{h_G(\bar
    t_i)}{\|h_G(\bar t_i)\|} \to 0$. It remains to show that this
  implies $v^T\beta_G(\bar x) = \bold 0^T$. We reason by contradiction
  and assume $v^T\beta_G(\bar x) \neq \bold 0^T$. Without loss of
  generality, we can assume that the limit of
  $\frac{h_G(\bar t_i)}{\|h_G(\bar t_i)\|}$, denoted
  $\zeta \in \real^n$, exists (this can be done because
  $\{ \frac{h_G(\bar t_i)}{\|h_G(\bar t_i)\|} \}$ is a sequence from
  the set of unit vectors in $\real^n$, which is compact, so there
  exists a convergent subsequence~\cite[Thm. 3.6]{WR:76}). This and
  the continuity of $\beta_G$ imply that
  $D'(\bar t_i) \to v^T\beta_G(\bar x)\zeta = 0$. Without loss of
  generality, assume
  $\frac{\|h_G(\bar t_i)\|}{\|h_G(\bar t_{i+1})\|} \to \infty$ (that
  this does not undermine generality is shown by
  Lemma~\ref{lem:sequences}(i)). Now, Lemma~\ref{lem:sequences}(ii)
  applied element-wise gives
  \begin{align}
    \frac{h_G(\bar t_i) - h_G(\bar t_{i+1})}{\|h_G(\bar t_i)\| -
    \|h_G(\bar t_{i+1})\|} \to \zeta. \label{eq:diffRatio} 
  \end{align}
  The sequence in \eqref{eq:diffRatio} can be written as
  \begin{equation}\label{eq:split1}
    \frac{h_G(\bar t_{i+1}) - h_G(\bar t_i)}{\bar t_{i+1} - \bar
      t_{i}} \frac{\bar t_{i+1} - \bar t_{i}}{\|h_G(\bar t_{i+1})\| -
      \|h_G(\bar t_i)\|}. 
  \end{equation}
  Using the continuous differentiability of $\nabla h$ and $G$ at
  $\bar x$, the first term of \eqref{eq:split1}
  \begin{align*}
    \frac{h_G(\bar t_{i+1}) - h_G(\bar t_i)}{\bar t_{i+1} - \bar
    t_{i}} \to \left . \dt(h_G(t)) \right |_{t = 0} 
    = (v^T\beta_G(\bar x) )^T \neq \bold 0,
  \end{align*}
  by hypothesis of contradiction. Consequently, the second term
  in~\eqref{eq:split1} converges to a non-zero scalar, which we denote
  by $a$.  Therefore, $\zeta = a(v^T\beta_G(\bar x))^T$. This implies
  that
  $D(\bar t_i) \to v^T\beta_G(\bar x)\zeta = a\|v^T\beta_G(\bar x)\|^2
  \neq 0$, which is a contradiction.
\end{proof}

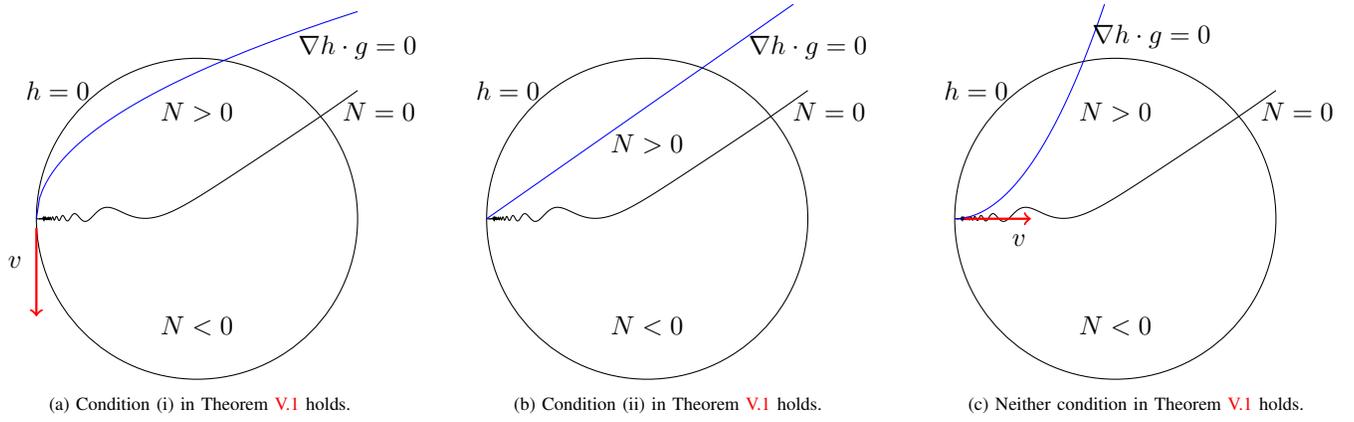
\begin{figure*}[ht!]
  \centering
  \begin{subfigure}{0.3\linewidth}
    \centering
    \begin{tikzpicture}
      \begin{axis}[ticks=none, axis lines=none,
        ymin = -0.1,
        ymax = 0.1,
        unit vector ratio={1 1}
        ]
        \draw (axis cs:0.075,0) circle [blue, radius=0.075];
        \node (source) at (axis cs: 0,0){};
        \node (destination) at (axis cs: 0,-0.05){};
        \draw[->,line width=0.3mm, color = red](source)--(destination);
        \addplot [
        domain=0.0001:0.15,
        samples=1000, 
        color=black,
        ]
        {0.1*(\x*(sin(deg(1/(4*x))))+20*x^2)};
        \addplot [
        domain=0:0.15,
        samples=100,
        color=blue,
        ]
        {0.25*x^(1/2)};
        \node[] at (axis cs: -0.01,-0.02){$v$};
        \node[] at (axis cs: 0.16,0.05){$N = 0$};
        \node[] at (axis cs: 0.15,0.08){$\nabla h\cdot g = 0$};
        \node[] at (axis cs: 0.075,-0.05){$N < 0$};
        \node[] at (axis cs: 0.075, 0.05){$N > 0$};
        \node[] at (axis cs: 0.01, 0.06){$h = 0$};
      \end{axis}
    \end{tikzpicture}
    \caption{\scriptsize Condition (i) in
      Theorem~\ref{thm:mimoBoundedness} holds.}
    \label{fig:2infty}
  \end{subfigure}
  \hfill
  \begin{subfigure}{0.3\linewidth}
    \centering
    \begin{tikzpicture}
      \begin{axis}[ ticks=none, axis lines=none,
        ymin = -0.1,
        ymax = 0.1,
        unit vector ratio={1 1}
        ]
        \draw (axis cs:0.075,0) circle [blue, radius=0.075];
        \addplot [
        domain=0.0001:0.15,
        samples=1000, 
        color=black,
        ]
        {0.1*(\x*(sin(deg(1/(4*x))))+20*x^2)};
        \addplot [
        domain=0:0.15,
        samples=100,
        color=blue,
        ]
        {0.7*x};
        \node[] at (axis cs: 0.16,0.05){$N = 0$};
        \node[] at (axis cs: 0.15,0.08){$\nabla h\cdot g = 0$};
        \node[] at (axis cs: 0.075,-0.05){$N < 0$};
        \node[] at (axis cs: 0.075, 0.035){$N > 0$};
        \node[] at (axis cs: 0.01, 0.06){$h = 0$};
      \end{axis}
    \end{tikzpicture}
    \caption{\scriptsize Condition (ii) in
      Theorem~\ref{thm:mimoBoundedness} holds.}
    \label{fig:bounded}
  \end{subfigure}
  \hfill
  \begin{subfigure}{0.3\linewidth}
    \centering
    \begin{tikzpicture}
      \begin{axis}[ ticks=none, axis lines=none,
        ymin = -0.1,
        ymax = 0.1,
        unit vector ratio={1 1}
        ]
        \draw (axis cs:0.075,0) circle [blue, radius=0.075];
        \addplot [
        domain=0.0001:0.15,
        samples=1000, 
        color=black,
        ]
        {0.1*(\x*(sin(deg(1/(4*x))))+20*x^2)};
        \addplot [
        domain=0:0.15,
        samples=100,
        color=blue,
        ]
        {100*x^4+20*x^2};
        \node[] at (axis cs: 0.16,0.05){$N = 0$};
        \node[] at (axis cs: 0.092,0.085){$\nabla h\cdot g = 0$};
        \node[] at (axis cs: 0.075,-0.05){$N < 0$};
        \node[] at (axis cs: 0.075, 0.05){$N > 0$};
        \node[] at (axis cs: 0.01, 0.06){$h = 0$};
        \node (source) at (axis cs: 0,0){};
        \node (destination) at (axis cs: 0.04,0){};
        \draw[->,line width=0.3mm, color = red](source)--(destination);
        \node[] at (axis cs: 0.03,-0.01){$v$};
      \end{axis}
    \end{tikzpicture}
    \caption{\scriptsize Neither condition in
      Theorem~\ref{thm:mimoBoundedness} holds.}
    \label{fig:undetermined}
  \end{subfigure}
  \caption{Illustration of the conditions identified in
    Theorem~\ref{thm:mimoBoundedness}. In~(a), $v$ solves
    \eqref{eq:mimoLinearTest} with $c_1 = 0$ and $c_2 < 0$. In~(b),
    there is no vector $v$ that points to the region in $\Cc$ where
    $N<0$ and is also tangential to the curve
    $\{\nabla h \cdot g = 0\}$.  In~(c), $v$ solves
    \eqref{eq:mimoLinearTest} with $c_1 > 0$ but with $c_2 = 0$, and
    thus neither condition in Theorem~\ref{thm:mimoBoundedness} is
    satisfied.}
  \label{fig:thm1}
\end{figure*}

Theorem~\ref{thm:mimoBoundedness} provides sufficient
conditions for boundedness of the min-norm controller at a point of
possible discontinuity.  Note that the second row of the the matrix
$A$ in~\eqref{eq:mimoLinearTest} is the gradient of
$\nabla h(x)f(x) + \alpha(h(x))$. Similarly, the $(2+i)^{\text{th}}$
row is the gradient of $\nabla h(x)g_{i}(x)$. Each of the two
equations $\nabla h(x)f(x) + \alpha(h(x)) = 0$ and
$\nabla h(x)g_i(x) = 0$ defines a differentiable $(n-1)$-dimensional
surface embedded in $\real^n$. Thus, the existence of a solution $v$ for
  \eqref{eq:mimoLinearTest} with $c_1 > 0$ and $c_2 < 0$ amounts to
  the existence of a vector that
  \begin{enumerate}
  \item points to the region in $\Cc$ that requires non-zero control
    for safety, and
  \item is perpendicular to the surfaces defined by
    $\nabla h(x)G(x) = 0$.
\end{enumerate}
This provides with a geometric intuition for the conditions identified
in Theorem~\ref{thm:mimoBoundedness}. Figure~\ref{fig:thm1}(a)-(b)
illustrates them for a generic two-dimensional single-input system.

We note that condition (ii) (with $c_2 \le 0$) in
Theorem~\ref{thm:mimoBoundedness} is \emph{almost} a negation of
condition (i) (with $c_2<0$).  This shows that (i) is almost a
sufficient and necessary condition for unboundedness of $u^*$. The gap
between both conditions stems from the fact that L'H\^opital's rule is
indeterminate when both derivatives of the numerator and the
denominator approach $0$. A geometric interpretation of this situation
is depicted in Figure~\ref{fig:undetermined}.


\begin{corollary}[Condition for Boundedness of Min-Norm Controller on
  $\Cc $]
  If condition (ii) in Theorem~\ref{thm:mimoBoundedness} holds for all
  $\bar x \in \Zc$, then $u^*$ is bounded on $\Cc $.
\end{corollary}

We revisit now Examples~\ref{ex:doubleIntegrator1}
and~\ref{ex:unboundedEg} in light of the above results.  Notice that
in both cases $\Zc = \{(1,0),(-1,0)\}$. Taking $\alpha$ with
$\alpha '(0) = 1$, \eqref{eq:mimoLinearTest} at $\bar x = (1,0)$
becomes
%
\[
  -2 \begin{bmatrix}
       1&0\\
       1&1\\
       0&d
  \end{bmatrix}v = \begin{bmatrix}c_1\\c_2\\0\end{bmatrix},
\]
where $d = 1$ for Example~\ref{ex:doubleIntegrator1} and $d = 0$ for
Example~\ref{ex:unboundedEg}.  It is clear that the only possible
solution for this system of equations with $c_1 \geq 0$ and
$c_2 \leq 0$ with $d = 1$ is the trivial solution $v =
\mathbf{0}$. Thus, by Theorem~\ref{thm:mimoBoundedness}(ii), $u_1^*$
from Example~\ref{ex:doubleIntegrator1} is bounded as its argument
approaches $\bar x$, as we would expect by our analysis of
Example~\ref{ex:doubleIntegrator1}. However, a solution $v = (0,1)$
solves the system with $d = 0$, $c_1 = 0 \geq 0$ and
$c_2 = -2 < 0$. By Theorem~\ref{thm:mimoBoundedness}(i), $u_2^*$
from Example~\ref{ex:unboundedEg} goes unbounded as it approaches
$\bar x$ from the direction of $v$, which is tangential to $\Cc
$. This is also expected by our analysis of
Example~\ref{ex:unboundedEg}.


  %
\begin{remark}[When Unbounded Min-Norm Is Inevitable]\label{rem:inevitable}
  {\rm The system of linear equations in \eqref{eq:mimoLinearTest} has
    a coefficient matrix $A$ with $m+2$ rows and $n$ columns. A
    non-trivial solution $v$ to \eqref{eq:mimoLinearTest} exists if
    the first two rows of $A$ are linearly independent and the
    remaining rows are linearly independent of the first two. This
    shows that, if the system data is such that the matrix $A$
    satisfies these independence properties, then an unbounded
    min-norm controller is inevitable.}
  %
  %
  \oprocend
\end{remark}



\section{Conclusions}
We have studied the continuity and boundedness properties of the
min-norm safe feedback controller for general control-affine systems
within the framework of control barrier functions (CBF).  After
re-interpreting the known results in the literature in light of the
notion of strong and weak CBFs, we have characterized the set of
possible points of discontinuity of the minimum-norm safe controller
and shown that it only depends on the safe set (and not on the
specific CBF or the sensitivity to the violation of the CBF
condition).  Based on this characterization, we have generalized the
known conditions to guarantee the continuity of the min-norm safe
controller and identified sufficient conditions for its
(un)boundedness.  Our results have important implications for the
synthesis of safe feedback controllers subject to hard constraints on
control effort.  Future work will explore questions about the
existence of continuous safe controllers when the min-norm controller
is discontinuous but bounded, the modification of CBFs that admit safe
controllers when no control bounds are present to incorporate such
limits, and the design of discontinuous (but bounded) safe
controllers.

\section*{Acknowledgments}
Mohammed Alyaseen would like to thank Pol Mestres for pointing out
Remark~\ref{re:limiting}(i). This work was partially supported by NSF
Award RI IIS-2007141.

\bibliography{../bib/alias,../bib/Main-add}

\begin{thebibliography}{10}

\bibitem{MN:42}
M.~Nagumo.
\newblock Über die {L}age der {I}ntegralkurven gewöhnlicher
  {D}ifferentialgleichungen.
\newblock {\em Proceedings of the Physico-Mathematical Society of Japan},
  24:551--559, 1942.

\bibitem{PW-FA:07}
P.~Wieland and F.~Allg{\"o}wer.
\newblock Constructive safety using control barrier functions.
\newblock {\em IFAC Proceedings Volumes}, 40(12):462--467, 2007.

\bibitem{XX-PT-JWG-ADA:15}
X.~Xu, P.~Tabuada, J.~W. Grizzle, and A.~D. Ames.
\newblock Robustness of control barrier functions for safety critical control.
\newblock {\em IFAC-PapersOnLine}, 48(27):54--61, 2015.

\bibitem{ADA-SC-ME-GN-KS-PT:19}
A.~D. Ames, S.~Coogan, M.~Egerstedt, G.~Notomista, K.~Sreenath, and P.~Tabuada.
\newblock Control barrier functions: theory and applications.
\newblock In {\em {E}uropean {C}ontrol {C}onference}, pages 3420--3431, Naples,
  Italy, 2019.

\bibitem{WX-CGC-CB:23}
W.~Xiao, C.~G. Cassandras, and C.~Belta.
\newblock {\em Safe Autonomy with Control Barrier Functions: Theory and
  Applications}.
\newblock Synthesis Lectures on Computer Science. Springer, New York, 2023.

\bibitem{RK-ADA-SC:21}
R.~Konda, A.~D. Ames, and S.~Coogan.
\newblock Characterizing safety: minimal control barrier functions from scalar
  comparison systems.
\newblock {\em IEEE Control Systems Letters}, 5(2):523--528, 2021.

\bibitem{BJM-MJP-ADA:15}
B.~J. Morris, M.~J. Powell, and A.~D. Ames.
\newblock Continuity and smoothness properties of nonlinear optimization-based
  feedback controllers.
\newblock In {\em {IEEE} Conf.\ on Decision and Control}, pages 151--158,
  Osaka, Japan, 2015.

\bibitem{GS:18}
G.~Still.
\newblock Lectures on {P}arametric {O}ptimization: {A}n {I}ntroduction.
\newblock {\em Preprint, Optimization Online}, 2018.

\bibitem{WWH:73}
W.~W. Hogan.
\newblock Point-to-set maps in mathematical programming.
\newblock {\em SIAM Review}, 15(3):591--603, 1973.

\bibitem{AVF-YI:90}
A.~V. Fiacco and Y.~Ishizuka.
\newblock Sensitivity and stability analysis for nonlinear programming.
\newblock {\em Annals of Operations Research}, 27(1):215--235, 1990.

\bibitem{SMR:82}
S.~M. Robinson.
\newblock Generalized equations and their solutions, part {II}: {A}pplications
  to nonlinear programming.
\newblock In {\em Optimality and Stability in Mathematical Programming}, pages
  200--221. Springer, Berlin, Heidelberg, 1982.

\bibitem{WSC-DVD:20}
W.~S. Cortez and D.~V. Dimarogonas.
\newblock Correct-by-design control barrier functions for {E}uler-{L}agrange
  systems with input constraints.
\newblock In {\em {A}merican {C}ontrol {C}onference}, pages 950--955, Denver,
  CO, USA, 2020.

\bibitem{DRA-DP:21}
D.~R. Agrawal and D.~Panagou.
\newblock Safe control synthesis via input constrained control barrier
  functions.
\newblock In {\em {IEEE} Conf.\ on Decision and Control}, pages 6113--6118,
  Austin, TX, USA, 2021.

\bibitem{AC:21a}
A.~Clark.
\newblock Verification and synthesis of control barrier functions.
\newblock In {\em {IEEE} Conf.\ on Decision and Control}, pages 6105--6112,
  Austin, TX, USA, 2021.

\bibitem{JJC-DL-KS-CJT-SLH:21}
J.~J. Choi, D.~Lee, K.~Sreenath, C.~J. Tomlin, and S.~L. Herbert.
\newblock Robust control barrier-value functions for safety-critical control.
\newblock In {\em {IEEE} Conf.\ on Decision and Control}, pages 6814--6821,
  Austin, TX, USA, 2021.

\bibitem{SB-LV:09}
S.~Boyd and L.~Vandenberghe.
\newblock {\em Convex Optimization}.
\newblock Cambridge University Press, Cambridge, UK, 2009.

\bibitem{QN-KS:16}
Q.~Nguyen and K.~Sreenath.
\newblock Exponential control barrier functions for enforcing high
  relative-degree safety-critical constraints.
\newblock In {\em {A}merican {C}ontrol {C}onference}, pages 322--328, Boston,
  MA, 2016.

\bibitem{WR:76}
W.~Rudin.
\newblock {\em Principles of Mathematical Analysis}.
\newblock McGraw-Hill, 1976.

\bibitem{RL:75}
R.~Le\ ao~de Andrade.
\newblock Complete convex hypersurfaces of a {H}ilbert space.
\newblock {\em Journal of Differential Geometry}, 10(4):491--499, 1975.

\bibitem{NB:23}
N.~Boumal.
\newblock {\em An {I}ntroduction to {O}ptimization on {S}mooth {M}anifolds}.
\newblock Cambridge University Press, Cambridge, UK, 2023.

\bibitem{MS:95}
M.~Spivak.
\newblock {\em Calculus on Manifolds}.
\newblock Addison-Wesley Publishing Company, 1995.

\bibitem{EG-DJ:85}
E.~Gilbert and D.~Johnson.
\newblock Distance functions and their application to robot path planning in
  the presence of obstacles.
\newblock {\em IEEE Journal on Robotics and Automation}, 1(1):21--30, 1985.

\bibitem{HHS:03}
H.~H. Sohrab.
\newblock {\em Basic Real Analysis}.
\newblock Birkh{\"a}user, Boston, MA, 2003.

\bibitem{AET:52}
A.~E. Taylor.
\newblock L'{H}ospital's rule.
\newblock {\em The American Mathematical Monthly}, 59(1):20--24, 1952.

\end{thebibliography}
\bibliographystyle{unsrt}

\appendices

\section{}
The next results are exploited in the proof of
Theorem~\ref{thm:mimoBoundedness}.

\begin{lemma}[Basic facts on real sequences]\label{lem:sequences}
  The following facts hold:
  \begin{enumerate}
  \item Any sequence $\{a_n\} \subset \realpos$ convergent to $ 0$
    contains a subsequence $\{\bar a_n\}$ with
    $\frac{\bar a_n}{\bar a_{n+1}} \to \infty$.
  \item If the sequences $\{a_n\}, \{b_n\} \subset \real$ both
    converge to $0$, $\frac{a_n}{b_n} \to L$, and
    $\frac{b_n}{b_{n+1}} \to c \neq 1$ (not excluding $c = \infty$),
    then $\frac{a_{n} - a_{n+1}}{b_n - b_{n+1}} \to L$.
  \end{enumerate}
\end{lemma}
\begin{proof}
  To prove \emph{(i)}, the subsequence $\{\bar a_n\}$ can be
  constructed as follows. Take $\bar a_1 = a_1$. By definition of
  convergence, for any $\bar a_n > 0$, there is
  $a_m \leq \bar a_n / n$. Taking $\bar a_{n+1} = a_m$ gives
  $\frac{\bar a_n}{\bar a_{n+1}} \geq n$.  Statement \emph{(ii)}
  follows directly from noting that
  $\frac{a_{n} - a_{n+1}}{b_n - b_{n+1}} - \frac{a_n}{b_n} = \left
    (\frac{a_n}{b_n} - \frac{a_{n+1}}{b_{n+1}}\right )$
  $ /\left ({\frac{b_n}{b_{n+1}} - 1}\right ) \to 0$.
\end{proof}



The following generalized version of L'H\^opital's rule is convenient
for our purposes.

\begin{lemma}[{Generalized
    L'H\^opital~\cite[Thm. II]{AET:52}}]\label{lem:genLhopital}
  Let the functions $f,g: (a,b) \to \real$ be continuously
  differentiable on $(a,b)$, with neither $g$ nor $g$ vanishing on
  $(a,b)$. Then
  $\liminf_{t \to a^+}\frac{f'(t)}{g'(t)} \leq \liminf_{t \to
    a^+}\frac{f(t)}{g(t)}\leq \limsup_{t \to a^+}\frac{f(t)}{g(t)}\leq
  \limsup_{t \to a^+}\frac{f'(t)}{g'(t)}$.
\end{lemma}

The following result shows a key property in the technical
argumentation of the proof of Theorem~\ref{thm:mimoBoundedness}.

\begin{lemma}\label{lem:halfThm1}
  Under the assumptions of Theorem~\ref{thm:mimoBoundedness}, let
  $v\in \real^n$ be such that
  $\limsup_{t \to 0^+} \|u^*(\bar x_t)\|= \infty$ (recall
  $\bar x_t = \bar x + vt$) and $v^T\beta_f(\bar x) \leq 0$. Then,
  there exists a sequence $\{t_i\} \to 0^+$ with $\{\bar x_{ t_i}\} \subset \Dc_-$ such that
  $v^T \beta_G(\bar x_{t_i}) \frac{(\nabla h(\bar x_{t_i})G(\bar
    x_{t_i}))^T}{\|\nabla h(\bar x_{t_i})G(\bar x_{t_i})\|} \to 0$.
\end{lemma}
\begin{proof}
  We utilize the abbreviations introduced at the beginning of the
  proof of Theorem~\ref{thm:mimoBoundedness} for convenience.  We
  consider the cases $v^T\beta_f(\bar x) < 0$ and
  $v^T\beta_f(\bar x) = 0$ separately.

  Case~1: If $v^T\beta_f(\bar x) < 0$, then
  by~\cite[Thm. 5.11]{WR:76}, $\bar x_t \in \Dc_-$ for sufficiently
  small~$t$.  Hence,
  $\limsup_{t \to 0^+}\|u^*(\bar x_t)\| = \lim_{t \to
    0^+}\frac{-N(t)}{D(t)} = \infty$.  Direct evaluation gives a
  $\frac{0}{0}$ type of limit. By Lemma~\ref{lem:genLhopital},
  $\limsup_{t \to 0^+}\frac{-{d}/{dt}(N(t))}{{d}/{dt}(D(t))} =
  \infty.$ Since $\dt(N(t)) = v^T\beta_f(\bar x_t)$ is continuous at
  $t = 0$, it is bounded on a small enough interval
  $t \in [0,\epsilon]$. Thus, the only way the $\limsup$ approaches
  $\infty$ is that there exists a sequence $\{t_i\}$ such that
  $\dt(D(t))|_{t = t_i} = v^T \beta_G(\bar x_{t_i})
  \frac{h_G(t_i)}{\|h_G(t_i)\|} \to
  0$.

  Case~2: If $v^T\beta_f(\bar x) = 0$, then for small enough positive
  $t$, either $\bar x_t \in \Dc_+$, $\bar x_t \in \Dc_-$, or
  $\bar x_t$ alternates between $\Dc_-$ and $\Dc_+$ indefinitely.  The
  first case is impossible if $u^*$ goes unbounded as $t \to 0^+$. The
  second case can be handled analogously to Case~1. Hence, we focus on
  the last case, where $\bar{x}_t$ alternates between $\Dc_-$ and
  $\Dc_+$ indefinitely as $t \to 0^+$.  This means that the continuous
  function $N$ approaches $0$ by alternating between positive and
  negative values indefinitely as $t \to 0^+$. By assumption, there
  exists $\{t_i\}$ such that $ \|u^*(\bar x_{t_i})\| \to
  \infty$. Without loss of generality, we assume that
  $\bar x_{t_i} \in \Dc_-$ for all $i$ and that $u^*(\bar x_{t_i})$ grows
  monotonically (a subsequence satisfying these assumptions can always
  be found). By continuity of $N(t)$ and the intermediate value
  theorem~\cite[Thm. 4.23]{WR:76}, for every $t_i$ there is an
  interval $(t_{1,i},t_{2,i})$ such that $\bar x_{t} \in \Dc_-$ for
  all $t \in (t_{1,i},t_{2,i})$ and $N(t_{1,i}) = N(t_{2,i}) = 0$. We
  distinguish two cases.

  Case~2.1: Assume there exists $\bar n$ such that
  $D(t_{1,i})D(t_{2,i}) \neq 0$ for all $i \ge \bar n$.  Thus,
  $\frac{N(t_{1,i})}{D(t_{1,i})} = \frac{N(t_{2,i})}{D(t_{2,i})} =
  0$. This together with the continuity of $N(t)/D(t)$ in
  $(t_{1,i},t_{2,i})$ implies that there exists
  $\bar t_i \in (t_{1,i},t_{2,i})$ where $-N(t)/D(t)$ attains its
  maximum in $(t_{1,i},t_{2,i})$. Therefore, the sequence
  $-\frac{N(\bar t_i)}{D(\bar t_i)}$ approaches $\infty$ since
  $\frac{- N(\bar t_i)}{D(\bar t_i)} \geq \frac{- N(t_i)}{D(t_i)} \to
  \infty$. The continuous differentiability of $-N(t)/D(t)$ in
  $(t_{1,i},t_{2,i})$ for all $i$, ensured by the lemma's assumptions,
  implies that
  $0 = \dt \big ( \frac{-N(t)}{D(t)} \big )\big |_{t = \bar t_i} =
  \frac{N'(\bar t_i)D(\bar t_i) - N(\bar t_i)D'(\bar t_i)}{D(\bar
    t_i)^2}$.  Keeping in mind that $\bar{x}_{\bar t_i} \in \Dc_-$ and
  thus $D(\bar{t}_i) \neq 0$, it follows that
  $N'(\bar t_i) = D'(\bar t_i)\frac{N(\bar t_i)}{D(\bar t_i)}$. Now
  since $\frac{N(\bar t_i)}{D(\bar t_i)} \to -\infty$ but
  $N'(t) = v^T\beta_f(\bar x_t) \to 0$, it should be that
  $D'(\bar t_i) = v^T\beta_G(\bar x_{\bar t_i})\frac{h_G(\bar
    t_i)}{\|h_G(\bar t_i)\|} \to 0$, which proves the statement.
  
  Case~2.2: Assume that for all $n$, there exists $i > n$ such that
  $D(t_{1,i})D(t_{2,i}) = 0$.
  Without loss of generality, assume $D(t_{1,i}) = 0$ for all $i$ (the
  same reasoning can be applied when $D(t_{2,i}) = 0$ for all $i$ or
  when this alternates). By assumption, both functions $N$ and $D$ are
  continuous on $[t_{1,i},t_i]$ and differentiable on
  $(t_{1,i},t_i)$. Using~\cite[Theorem 5.9]{WR:76}, for each $i$,
  there exists $\bar t_i$ such that
  $N'(\bar t_i) = \frac{N(t_i) - N(t_{1,i})}{D(t_i) -
  D(t_{1,i})}D'(\bar t_i) = \frac{N(t_i)}{D(t_i)}D'(\bar t_i)$.  A
  reasoning similar to that of Case~2.1 now yields
  $D'(\bar t_i) \to 0$.
\end{proof}

\end{document}